\documentclass[a4paper,10pt]{article}
\usepackage{amssymb,amsmath,amsthm}
\usepackage{fullpage}
\usepackage{hyperref}
\usepackage{eucal}
\usepackage{color}
\usepackage{stackrel}
\usepackage{graphicx}
\newcommand{\ie}{\emph{i.e.}}
\newcommand{\eg}{\emph{e.g.}}
\newcommand{\cf}{\emph{cf.}}
\newcommand{\Real}{\mathbb{R}}
\newcommand{\Com}{\mathbb{C}}

\newcommand{\Int}{\mathbb{Z}}

\newcommand{\sgn}{\mathop{\mathrm{sgn}}\nolimits}

\newcommand{\rot}{\mathop{\mathrm{rot}}\nolimits}
\newcommand{\Dom}{\mathsf{D}}
\newcommand{\dist}{\mathop{\mathrm{dist}}\nolimits}

\newcommand{\vertiii}[1]{{\left\vert\kern-0.25ex\left\vert\kern-0.25ex\left\vert #1 
    \right\vert\kern-0.25ex\right\vert\kern-0.25ex\right\vert}}

\newcommand{\sii}{L^2}

\newcommand{\der}{\mathrm{d}}
\newtheorem{Theorem}{Theorem}
\newtheorem{Lemma}{Lemma}
\newtheorem{Proposition}{Proposition}

\theoremstyle{definition}
\newtheorem{Remark}{Remark}

\usepackage[normalem]{ulem}
\definecolor{DarkGreen}{rgb}{0,0.5,0.1} 

\newcommand\soutD{\bgroup\markoverwith
{\textcolor{DarkGreen}{\rule[.5ex]{2pt}{1pt}}}\ULon}
\newcommand\soutP{\bgroup\markoverwith
{\textcolor{blue}{\rule[.5ex]{2pt}{1pt}}}\ULon}
\newcommand{\Hm}[1]{\leavevmode{\marginpar{\tiny%
$\hbox to 0mm{\hspace*{-0.5mm}$\leftarrow$\hss}%
\vcenter{\vrule depth 0.1mm height 0.1mm width \the\marginparwidth}%
\hbox to
0mm{\hss$\rightarrow$\hspace*{-0.5mm}}$\\\relax\raggedright #1}}}

\begin{document}
%
\title{\textbf{\Large
Absence of eigenvalues of two-dimensional 
magnetic Schr\"odinger operators
}}
\author{Luca Fanelli, David Krej\v{c}i\v{r}{\'\i}k and Luis Vega}
\date{\small 
\emph{
\begin{quote}
\begin{itemize}
\item[$a)$] 
Dipartimento di Matematica, SAPIENZA Universit\`a di Roma,
P.~le Aldo Moro 5, 00185 Roma;
fanelli@mat.uniroma1.it.%
\item[$b)$] 
Department of Mathematics, Faculty of Nuclear Sciences and 
Physical Engineering, Czech Technical University in Prague, 
Trojanova 13, 12000 Prague 2, Czechia;
david.krejcirik@fjfi.cvut.cz.%
\\
\item[$c)$]
Departamento de Matem\'aticas, Universidad del Pais Vasco, 
Aptdo.~644, 48080 Bilbao, \&
Basque Center for Applied Mathematics (BCAM), 
Alameda Mazarredo 14, 48009 Bilbao, Spain; 
luis.vega@ehu.es \& lvega@bcamath.org.
\end{itemize}
\end{quote}
}
\smallskip
17 October 2017}
\maketitle

\begin{abstract}
\noindent
By developing the method of multipliers,
we establish sufficient conditions on the electric potential
and magnetic field
which guarantee that the corresponding two-dimensional 
Schr\"odinger operator possesses no point spectrum.
The settings of complex-valued electric potentials
and singular magnetic potentials of Aharonov-Bohm field 
are also covered.
\end{abstract}
%

%

\section{Introduction}
%
Given a vector potential $A : \Real^2 \to \Real^2$
and a scalar potential $V :\Real^2 \to \Real$,
we consider the Schr\"odinger operator
\begin{equation}\label{operator}
  H_{A,V} := (-i\nabla+A)^2 + V
  \qquad \mbox{in} \qquad \sii(\Real^2)
  \,.
\end{equation}
It is the quantum Hamiltonian
of a non-relativistic electron interacting 
with the (vector) electric field $E=-\nabla V$
and the (scalar) magnetic field
\begin{equation}\label{Hodge.2D}
  B = \partial_1 A_2 - \partial_2 A_1
  \,.
\end{equation}

Writing $V= V_+ - V_-$ with $V_\pm$ being non-negative, 
we always assume that the negative part~$V_-$ 
of the electric potential~$V$
is small in a suitable sense (see~\eqref{Ass1}), 
in order to make~$H_{A,V}$ self-adjoint	
and bounded from below.
More specifically, we understand~$H_{A,V}$ as the Friedrichs extension
of the symmetric operator~\eqref{operator} 
initially defined on smooth compactly supported functions. 

The two-dimensional situation is special in the sense that
$\Real^2$ is the lowest dimensional Euclidean space for
which the addition of the magnetic potential is non-trivial,
while the unperturbed operator $H_{0,0}$ is still \emph{critical},
\ie\ unstable under small perturbations. 
In fact, it is well known that
the purely electric operator~$H_{0,V}$ possesses negative discrete 
eigenvalues whenever~$V$ is negative and compactly supported,
and there also exist examples of potentials generating
positive eigenvalues embedded in the essential spectrum.
At the same time, the spectrum
of the purely magnetic operator~$H_{A,0}$ can be quite general, 
ranging from the purely essential spectrum $[0,\infty)$
for compactly supported magnetic field~$B$ (Aharonov-Bohm solenoid),
through the discrete set of infinitely degenerated eigenvalues
(Landau levels) for constant $B\not=0$,
to purely discrete spectrum if~$|B|$ diverges at infinity (magnetic bottles).

The objective of this paper is to identify physically relevant 
conditions which guarantee 
a \emph{total absence of eigenvalues} of~$H_{A,V}$.
Because of the gauge invariance, 
these conditions should be given in terms of
the physical quantity~$B$ and not~$A$.
To state the desired result, 
we use the abbreviations $\nabla_{\!A} := \nabla+iA$ for the magnetic gradient,
$r(x):=|x|$ for the distance function from the origin
and $\partial_r f(x) := \frac{x}{|x|} \cdot \nabla f(x)$
for the radial derivative of a function~$f$.

\begin{Theorem}\label{Thm}
Let $A \in \sii_\mathrm{loc}(\Real^2;\Real^2)$ 
be such that $B \in L_\mathrm{loc}^2(\Real^2)$.
Suppose that $V \in L^1(\Real^2;\Real)$
admits the decomposition
$V = V^{(1)} + V^{(2)}$ 
with $V^{(1)} \in W^{1,1}_\mathrm{loc}(\Real^2)$ 
and $V^{(2)} \in L_\mathrm{loc}^2(\Real^2)$.
Assume that there exist numbers $b,b_1,b_2,b_3,b_4 \in [0,1)$ 
satisfying
\begin{equation}\label{Ass.b}
  b_1 +b_2^2 + b_3^2 + b_4 < 1
  \,,
\end{equation}
such that, for all $\psi \in C_0^\infty(\Real^2)$,
\begin{equation}\label{Ass1}
  \int_{\Real^2} 
  V_-(x) \, |\psi|^2 
  \leq b^2 \int_{\Real^2} |\nabla_{\!A} \psi|^2 
  \,, 
\end{equation}
and
\begin{equation}\label{Ass2}
\begin{aligned}
  \int_{\Real^2} 4 \, r^2 \, |B|^2 \, |\psi|^2 
  &\leq b_1^2 \int_{\Real^2} |\nabla_{\!A} \psi|^2 \,,
  &&&
  \int_{\Real^2} [\partial_r(r \, V^{(1)})]_+ \, |\psi|^2
  &\leq b_2^2 \int_{\Real^2} |\nabla_{\!A} \psi|^2 \,,
  \\
  \int_{\Real^2} |V^{(2)}| \, |\psi|^2
  &\leq b_3^2 \int_{\Real^2} |\nabla_{\!A} \psi|^2 \,,
  &&&
  \int_{\Real^2} 4 \, r^2 \, |V^{(2)}|^2 \, |\psi|^2
  &\leq b_4^2 \int_{\Real^2} |\nabla_{\!A} \psi|^2 \,,
\end{aligned}
\end{equation}
Then $H_{A,V}$ has no eigenvalues, \ie\
$\sigma_\mathrm{p}(H_{A,V}) = \varnothing$.
\end{Theorem}

The subordination condition~\eqref{Ass1} immediately implies
that no non-positive number (including zero) can be an eigenvalue of~$H_{A,V}$.
The interest of the theorem lies in~\eqref{Ass2} with~\eqref{Ass.b},
which is a sufficient condition to avoid the existence
of positive eigenvalues.  
The latter is clearly satisfied if $B=0$, $V^{(2)}=0$ 
and $\partial_r(r V^{(1)}) \leq 0$, 
where the last condition is a classical repulsiveness requirement.
If $B\not=0$, $V^{(2)}\not=0$ 
or $[\partial_r(r V^{(1)})]_+ \not= 0$, however,
it is not \emph{a priori} clear that~\eqref{Ass2} is not void. 
In fact, if there is no magnetic field (\ie~$B\not=0$)
and $V^{(2)}\not=0$ or $[\partial_r(r V^{(1)})]_+ \not= 0$,
the criticality of the two-dimensional free Hamiltonian~$H_{0,0}$
implies that \eqref{Ass2}~cannot be satisfied
(a similar statement holds for~\eqref{Ass1} with $V_-\not=0$).
If $V^{(2)}\not=0$ or $[\partial_r(r V^{(1)})]_+ \not= 0$,
it is therefore necessary that there is a magnetic field
to exclude the existence of eigenvalues via Theorem~\ref{Thm}.
The following proposition particularly ensures 
that~\eqref{Ass2} is generally not void.

\begin{Proposition}
If there exist numbers $b,b_1,b_2,b_3,b_4 \in [0,1)$ such that
\begin{equation}\label{suff1}
  V_- \leq \pm b^2 \, B \,,
\end{equation}
and
\begin{equation}\label{suff2}
\begin{aligned}
  4 \, r^2 \, |B|^2 
  &\leq \pm b_1^2 \, B
  \,,
  &&&
  [\partial_r(r \, V^{(1)})]_+  
  &\leq \pm b_2^2 \, B
  \,,
  \\
  |V^{(2)}|  
  &\leq \pm b_3^2 \, B 
  \,,
  &&&
  4 \, r^2 \, |V^{(2)}|^2  
  &\leq \pm b_4^2 \, B 
  \,,
\end{aligned}
\end{equation}
with either a plus or minus sign, 
then~\eqref{Ass1} and~\eqref{Ass2} hold. 
\end{Proposition}
\begin{proof}
The claim follows from the well-known lower bound
$$
  \forall \psi \in C_0^\infty(\Real^2) \,, \qquad
  \int_{\Real^2} |\nabla_{\!A} \psi|^2  
  \geq \int_{\Real^2} \pm B \, |\psi|^2  
  \,,
$$
which holds with either of the signs~$\pm$
(see, \eg, \cite{Balinsky-Laptev-Sobolev_2004}).
\end{proof}

For instance, if $V=0$ 
and the magnetic field $B$ is of definite sign,
sufficiently small in the supremum norm
and with a sufficiently small support, 
then~\eqref{suff1} and~\eqref{suff2}, 
and therefore~\eqref{Ass1} and~\eqref{Ass2}, hold. 

\begin{Remark}
Another source of sufficient conditions 
to guarantee~\eqref{Ass1} and~\eqref{Ass2}
are in principal magnetic Hardy-type inequalities,
reflecting the \emph{subcriticality} of~$H_{A,0}$ 
whenever $B \not= 0$. 
For continuous~$A$ satisfying the transverse gauge $x \cdot A(x) = 0$
and locally integrable~$B$,
the Laptev-Weidl magnetic Hardy-type inequality
(see~\cite{Laptev-Weidl_1999}) states
\begin{equation}\label{LW} 
  \forall \psi \in C_0^\infty(\Real^2) \,, \qquad
  \int_{\Real^2} |\nabla_{\!A} \psi(x)|^2 \, \der x
  \geq \int_{\Real^2} 
  \frac{\dist(\Phi_B(x),\Int)^2}{|x|^2} \, |\psi(x)|^2
  \, \der x
  \,,
\end{equation}
where
$$
  \Phi_B(x) := 
  \frac{1}{2\pi} \int_{\{|\xi|\leq |x|\}} B(\xi) \, \der \xi
$$ 
denotes the \emph{magnetic flux} through the sphere of radius~$|x|$
centred at the origin.
Unfortunately, the pointwise condition 
$
  4 |x|^2 |B(x)|^2 \leq |x|^{-2} \dist(\Phi_B(x),\Int)^2
$,
to guarantee the first inequality of~\eqref{Ass2} via~\eqref{LW},
can never be satisfied for all sufficiently small~$|x|$.
From~\eqref{LW} one can deduce the bound
\begin{equation} 
  \forall \psi \in C_0^\infty(\Real^2) \,, \qquad
  \int_{\Real^2} |\nabla_{\!A} \psi(x)|^2 \, \der x
  \geq \tilde{c}_{B} \int_{\Real^2} 
  \frac{|\psi(x)|^2}{1+|x|^2}
  \, \der x
  \,,
\end{equation}
where~$\tilde{c}_{B}$ is a constant,
which is positive if and only if 
the total magnetic flux $\lim_{|x|\to\infty}\Phi_B(x)$
is not an integer.
See also~\cite{Weidl_1999} for other types of magnetic
Hardy-type inequalities.
As the most recent result, 
the following global Hardy-type inequality 
was derived in~\cite{CK} for \emph{any} smooth~$A$ 
\begin{equation}\label{CK} 
  \forall \psi \in C_0^\infty(\Real^2) \,, \qquad
  \int_{\Real^2} |\nabla_{\!A} \psi(x)|^2 \, \der x
  \geq c_{B} \int_{\Real^2} 
  \frac{|\psi(x)|^2}{1+|x|^2\log^2|x|}
  \, \der x
  \,,
\end{equation}
where~$c_{B}$ is a constant, 
which is positive if and only if $B$~is not identically equal to zero.
The constant~$c_B$ is given in terms of the first Neumann eigenvalue of 
the magnetic Laplacian in a disk and it is not clear whether 
a bound of the type~\eqref{CK} can actually provide a useful sufficient 
condition to guarantee~\eqref{Ass2}.
In any case, such a condition clearly forces~$B$ 
to decay sufficiently fast at infinity. 
\end{Remark}

Our strategy to establish Theorem~\ref{Thm} is based on 
the \emph{method of multipliers} 
developed in the spectral context
for Schr\"odinger operators in three and higher dimensions 
in our preceding paper~\cite{FKV}. 
The restriction to the higher dimensions in~\cite{FKV}
was caused by the essential usage of the classical Hardy inequality
in the proof.
The present paper is enabled by the observation that 
while the classical Hardy inequality 
is valid in three and higher dimensions only,
it can be effectively replaced by the aforementioned
magnetic Hardy-type inequalities in two dimensions
and still yields a result in the context of multipliers. 
(We are grateful to Timo Weidl for the initial impetus
to think about this extension.)
Since the proof of Theorem~\ref{Thm} requires some important
modifications of the original ideas of~\cite{FKV}
and condition~\eqref{Ass2} differs from the form 
of sufficient conditions established in~\cite{FKV},
we have decided to present this two-dimensional result to the community.

The organisation of this paper is as follows.
Theorem~\ref{Thm} is proved in Section~\ref{Sec.proof},
where we simultaneously present its higher-dimensional analogue
(Theorem~\ref{Thm.multi}).
In Section~\ref{Sec.AB} we discuss the possibility
of extending the present results to complex-valued electric potentials;
in addition to the case of general regular fields
(Theorems~\ref{Thm.nsa1} and~\ref{Thm.robust}),
we establish results for the singular Aharonov-Bohm potential 
(Theorem~\ref{Thm.AB}). 
	
\section{The proof in any dimension}\label{Sec.proof}
%
We proceed in any dimension $d \geq 1$.
At the same time, following~\cite{FKV}, 
we also allow the electric potential~$V$ to be complex-valued.
Physically, the imaginary part of~$V$ can be interpreted
as an energy gain/loss in an open quantum system.
 
Let us therefore assume
$A \in \sii_\mathrm{loc}(\Real^d;\Real^d)$
and  
$V \in L^1_\mathrm{loc}(\Real^d;\Com)$.
In any dimension, the physically relevant quantity
is the $d$-covariant \emph{magnetic tensor}
$$
  B^* := \nabla A - (\nabla A)^T  
$$
and we assume $B^* \in \sii_\mathrm{loc}(\Real^d;\Real^{d \times d})$.
If $d=1$ this quantity is always equal to zero,
so it is reasonable to exclude the one-dimensional situation
from further considerations, but formally it is covered in the following	.
The \emph{magnetic field}~$B$ is the $(d-2)$-contravariant tensor
obtained from~$B^*$ as its Hodge dual. 
We therefore arrive at the scalar field~\eqref{Hodge.2D} for $d=2$ 
and at the usual vector field $B = \rot A$ for $d=3$.
We refer to~\cite{CK} for more details on the formalism
of the magnetic field in any dimension.

Let us consider the quadratic form
\begin{equation}\label{form} 
  h_{A,V}[\psi] := \int |\nabla_{\!A}\psi|^2
  + \int  V |\psi|^2
  \,, \qquad
  \Dom(h_{A,V}) := \overline{C_0^\infty(\Real^d)}^{\vertiii{\cdot}}
  \,,
\end{equation}
where 
\begin{equation}\label{Aspace}
  \vertiii{\psi}^2 := \int |\nabla_{\!A}\psi|^2
  + \int \Re V_+ \, |\psi|^2 + \int |\psi|^2
  \,.
\end{equation}
Here and in the sequel we abbreviate $\int := \int_{\Real^d}$
and omit the arguments of integrated functions.
Under the assumption that there exist numbers $a_1, a_2 \in [0,1)$ such that,
for every $\psi \in C_0^\infty(\Real^d)$, 
\begin{equation}\label{Ass1.bis}
  \int \Re V_- \, |\psi|^2  
  \leq a_1^2 \int  |\nabla_{\!A} \psi|^2 
  \qquad\mbox{and}\qquad
  \int |\Im V| \, |\psi|^2  
  \leq a_2^2 \int  |\nabla_{\!A} \psi|^2  
  \,,
\end{equation}
the form $h_{A,V}$ is sectorial and closed.	
Let us denote by~$H_{A,V}$ the m-sectorial operator
associated with~$h_{A,V}$ via the first representation theorem
(\cf~\cite[Thm.~VI.2.1]{Kato}).
We have $\Dom(H_{A,V}) \subset \Dom(h_{A,V}) \subset W^{1,2}(\Real^d)$.
Here the last inclusion employs the diamagnetic inequality
\begin{equation}\label{dia}
  \forall \psi \in W_\mathrm{loc}^{1,2}(\Real^d) 
  \,, \qquad
  \big|\nabla_{\!A}\psi\big| \geq \big|\nabla|\psi|\big|
  \,,
\end{equation}
which we shall frequently use in the sequel.
(If $\Im V = 0$, then~\eqref{Ass1.bis} coincides with~\eqref{Ass1} 
and in this case the operator~$H_{A,V}$ is self-adjoint 
and bounded from below.)

In view of~\eqref{Ass1.bis},
a vertex and a semi-angle of $h_{A,V}$ 
are given by~$0$ and~$\pi/4$, respectively.
Consequently, 
\begin{equation}\label{sector}
  \sigma(H_{A,V}) \subset \{\lambda \in \Com : \ \Re\lambda \geq |\Im\lambda|\}
\end{equation}
and in particular~$H_{A,V}$ has no complex eigenvalue~$\lambda$
with $\Re\lambda < |\Im\lambda|$.
Actually, also the presence of zero eigenvalue
can be excluded by~\eqref{Ass1.bis} with help of~\eqref{dia}. 

Hence, it remains to exclude the existence of eigenvalues of~$H_{A,V}$ 
in the sector on the right-hand side of~\eqref{sector}.
To this purpose, we employ the following crucial lemma.
\begin{Lemma}\label{Lem.crucial}
Let $A \in \sii_\mathrm{loc}(\Real^d;\Real^d)$ 
be such that $B^* \in L_\mathrm{loc}^2(\Real^d,\Real^{d \times d})$.
Suppose that $V \in L^1(\Real^d;\Com)$
admits the decomposition $\Re V = \Re V^{(1)} + \Re V^{(2)}$ 
with $\Re V^{(1)} \in W^{1,1}_\mathrm{loc}(\Real^d)$ 
and $\Re V^{(2)} \in L_\mathrm{loc}^2(\Real^d)$.
Assume also~\eqref{Ass1.bis}. 
Let~$u$ be a solution of $H_{A,V} u = \lambda u$ 
with $\Re\lambda \geq |\Im\lambda|$
satisfying
\begin{equation}\label{crucial.cond}
  \left( 
  r^2 \, |B_\tau^*|^2 + [\partial_r(r \, \Re V^{(1)})]_+
  + (1+r^2) \, |\Re V^{(2)}|^2
  + r^2 \, \Im V  
  + r \, \Re V_-
  + r^{-1}
  \right)
  |u|^2
  \in L^1(\Real^d)
  \,,
\end{equation}
where $B_\tau^* := \frac{x}{|x|} \cdot B^*$.
Then also
$
  \left(
  r \, |\nabla_{\!A} u^-|
  + [\partial_r(r \, \Re V^{(1)})]_-
  + r \, \Re V_+
  \right)
  \in L^1(\Real^d)
$
and the identity 
\begin{eqnarray}\label{crucial}
\lefteqn{
  \int |\nabla_{\!A} u^-|^2
  + \frac{|\Im\lambda|}{(\Re\lambda)^{1/2}}
  \left(
  \int r \, |\nabla_{\!A} u^-|^2
  - \frac{d-1}{2} \int\frac{|u^-|^2}{r}
  + \int r \, \Re V \, |u^-|^2
  \right)
} 
  \nonumber \\
  && = -2 \Im \int r \, B_\tau^* \cdot u^- \overline{\nabla_{\!A}u^-} 
  + \int \partial_r(r \, \Re V^{(1)}) \, |u^-|^2
  + (1-d) \int \Re V^{(2)} \, |u^-|^2
  \nonumber \\
  && \quad
  - 2 \Re \int r \, \Re V^{(2)} \, u^- \overline{\partial_r^{A}u^-} 
  + 2 \Im \int r \, \Im V \, u^- \overline{\partial_r^{A}u^-} 
\end{eqnarray}
holds true
(if $\lambda=0$ then the term multiplied by $|\Im\lambda|/(\Re\lambda)^{1/2}$
is not present), 
where $\partial_r^A f(x) := \frac{x}{|x|} \cdot \nabla_{\!A} f(x)$
and
$$
   u^\pm(x) := e^{\pm i \sgn(\Im\lambda) \;\! (\Re\lambda)^{1/2}|x|} \, u(x)
  \,.
$$
\end{Lemma}
\begin{proof}
The identity can be derived by following 
the method of multipliers developed in~\cite{FKV}
based on previous ideas 
of~\cite{Ikebe-Saito} and~\cite{Barcelo-Vega-Zubeldia_2013}.
Since the lemma is not explicitly stated in~\cite{FKV},
we sketch the proof. 

The eigenvalue equation $H_{A,V} u = \lambda u$ 
means that $u \in \Dom(H_{A,V})$ and 
\begin{equation}\label{weak}
  \forall v \in \Dom(h_{A,V}) 
  \,, \qquad
  - \int \overline{\nabla_{\!A} v} \, \nabla_{\!A} u 
  + \lambda \int \overline{v} \, u
  = \int \overline{v} \, V u
  \,.
\end{equation}
Let $G_1,G_2,G_3:\Real^d \to \Real$ be three smooth functions.
Choosing $v:=G_1 u$ in~\eqref{weak}, taking the real part of 
the obtained identity and integrating by parts, we obtain
\begin{equation}\label{G1}
  \Re\lambda \int G_1 \, |u|^2
  - \int G_1 \, |\nabla_{\!A}u|^2
  +\frac{1}{2} \int \Delta G_1 \, |u|^2
  = \int G_1 \, \Re V \, |u|^2
  \,.
\end{equation}
Analogously, choosing $v:=G_2 u$ in~\eqref{weak}, taking the imaginary part of 
the obtained identity and integrating by parts, we obtain
\begin{equation}\label{G2}
  \Im\lambda \int G_2 \, |u|^2
  - \Im \int \nabla G_2 \cdot \overline{u} \nabla_{\!A}u
  = \int G_2 \, \Im V \, |u|^2
  \,.
\end{equation}
Finally, choosing 
$
  v := [\Delta_A,G_3]
  = 2 \nabla G_3 \cdot \nabla_{\!A} u + \Delta G_3 \, u
$ 
in~\eqref{weak}
where $\Delta_A := \nabla_{\!A}\cdot\nabla_{\!A}$
is the magnetic Laplacian, 
taking the real part of 
the obtained identity, integrating by parts
and multiplying the result by $-1/2$, we obtain
\begin{multline}\label{G3}
  \int \nabla_{\!A}u \cdot \nabla^2 G_3 \cdot \overline{\nabla_{\!A} u}
  - \frac{1}{4} \int \Delta^2 G_3 \, |u|^2
  + \Im\lambda \ \Im\int \nabla G_3 \cdot u \overline{\nabla_{\!A} u}
  +\Im \int \nabla G_3 \cdot B^* \cdot u \overline{\nabla_{\!A} u}
  \\
  = -\frac{1}{2} \int \Delta G_3 \, \Re V \, |u|^2
  - \Re \int \nabla G_3 \cdot V \, u \overline{\nabla_{\!A} u}
  \,,
\end{multline}
where $\nabla^2 G_3$ denotes the Hessian matrix of~$G_3$
and $\Delta^2 := \Delta\Delta$ is the bi-Laplacian.
Identity~\eqref{crucial} is obtained by combining 
\eqref{G1}--\eqref{G3} with special choices of the multipliers:
\begin{multline}
  \big[\mbox{\eqref{G1} with } G_1(x) := 1\big] 
  + \big[\mbox{\eqref{G2} with } G_2(x) := 2 \, \Re\lambda \, \sgn(\Im\lambda) \, |x| \big]
  + \big[\mbox{\eqref{G3} with } G_3(x) := |x|^2\big]
  \\
  - \left[\mbox{\eqref{G1} with } G_1(x) 
  := \frac{|\Im\lambda|}{(\Re\lambda)^{1/2}} \, |x|\right]
  \,.
\end{multline}
(If $\lambda=0$ then the last subtraction is not performed.)
Here the main idea is to replace~$u$ by~$u^-$ using the identities 
$$
  \big|\nabla_{\!A} u^-(x)\big| = 
  \left|
  \nabla_{\!A} u(x) 
  - i \, (\Re\lambda)^{1/2} \, \sgn(\Im\lambda) \, \frac{x}{|x|} \, u(x)
  \right|
  \qquad \mbox{and} \qquad
  B_\tau^* \cdot \overline{u} \nabla_{\!A} u
  =  B_\tau^* \cdot \overline{u^-} \nabla_{\!A} u^-
  \,,
$$
where the latter employs the fact that $B_\tau^*$ is tangential,
\ie\ $x \cdot B_\tau^*(x) = 0$.

Up to now, the procedure explained above has been purely formal,
because we \emph{a priori} do not know that the individual integrals converge.
To make it rigorous, one can follow~\cite{FKV} and replace~$u$
by approximating solutions by using a standard cutoff and mollification argument. 
In this way, one arrives at an approximating version of~\eqref{crucial} 
and the desired identity is obtained after passing to the limit
in the cutoff and mollification parameters,
by employing the convergence of a set of integrals 
expressed by~\eqref{crucial.cond}.
\end{proof}

Now let us come back to the initial hypothesis that~$V$ is real-valued.
Then~$H_{A,V}$ is self-adjoint, necessarily $\Im\lambda=0$ 
and it remains to exclude the existence of non-negative eigenvalues. 
In this case, \eqref{crucial} reduces to
\begin{eqnarray}\label{crucial.ss}
  \int |\nabla_{\!A} u^-|^2
  &=& -2 \Im \int r \, B_\tau^* \cdot u^- \overline{\nabla_{\!A}u^-} 
  + \int \partial_r(r \, V^{(1)}) \, |u^-|^2
  \nonumber \\
  && 
  + (1-d) \int V^{(2)} \, |u^-|^2
  - 2 \Re \int r \, V^{(2)} \, u^- \overline{\partial_r^{A}u^-}
  \,. 
\end{eqnarray}
Using the Schwarz inequality, we have
\begin{eqnarray*}
  \int |\nabla_{\!A} u^-|^2
  &\leq& 2 \, \sqrt{\int r^2 \, |B_\tau^*|^2 \, |u^-|^2 } \, 
  \sqrt{\int |\nabla_{\!A}u^-|^2} 
  + \int [\partial_r(r \, V^{(1)})]_+ \, |u^-|^2
  \nonumber \\
  && 
  + |1-d| \int |V^{(2)}| \, |u^-|^2
  + 2 \, \sqrt{\int r^2 \, |V^{(2)}|^2 \, |u^-|^2}
  \, \sqrt{\int |\nabla_{\!A}u^-|^2}
  \\
  &\leq&
  (b_1 +b_2^2 + (d-1) \, b_3^2 + b_4) \int |\nabla_{\!A}u^-|^2
  \,, 
\end{eqnarray*}
where the last inequality follows by conditions
\begin{equation}\label{Ass2.bis}
\begin{aligned}
  \int 4 \, r^2 \, |B_\tau^*|^2 \, |\psi|^2 
  &\leq b_1^2 \int |\nabla_{\!A} \psi|^2 \,,
  &&&
  \int [\partial_r(r \, V^{(1)})]_+ \, |\psi|^2
  &\leq b_2^2 \int |\nabla_{\!A} \psi|^2 \,,
  \\
  \int |V^{(2)}| \, |\psi|^2
  &\leq b_3^2 \int |\nabla_{\!A} \psi|^2 \,,
  &&&
  \int 4 \, r^2 \, |V^{(2)}|^2 \, |\psi|^2
  &\leq b_4^2 \int |\nabla_{\!A} \psi|^2 \,,
\end{aligned}
\end{equation}
assumed to be valid for every $\psi \in C_0^\infty(\Real^d)$.
Making the hypothesis
\begin{equation}\label{Ass.b.bis}
  b_1 +b_2^2 + (d-1) \, b_3^2 + b_4 < 1
  \,,
\end{equation}
we conclude with $\nabla_{\!A} u^- = 0$.
By the diamagnetic inequality~\eqref{dia}, 
it follows that $u^- = 0$ and thus $u=0$.
Consequently, the eigenvalue equation $H_{A,V} u =\lambda u$
for $\lambda \geq 0$ admits only trivial solutions.   
It concludes the proof that the point spectrum of $H_{A,V}$ is empty.

Let us summarise the multidimensional result
into the following theorem.

\begin{Theorem}\label{Thm.multi}
Let $A \in L_\mathrm{loc}^2(\Real^d;\Real^d)$ 
be such that $B^* \in L_\mathrm{loc}^2(\Real^d;\Real^{d \times d})$.
Suppose that $V \in L^1(\Real^d;\Real)$
admits the decomposition $V = V^{(1)} + V^{(2)}$ 
where $V^{(1)} \in W^{1,1}_\mathrm{loc}(\Real^d)$ 
and $V^{(2)} \in L_\mathrm{loc}^2(\Real^d)$.
Assume that there exist numbers $a,b_1,b_2,b_3,b_4 \in [0,1)$ 
satisfying~\eqref{Ass.b.bis}
such that~\eqref{Ass1.bis} and~\eqref{Ass2.bis} hold. 
Then $H_{A,V}$ has no eigenvalues, \ie\
$\sigma_\mathrm{p}(H_{A,V}) = \varnothing$.
\end{Theorem}

Theorem~\ref{Thm} is a special case for $d=2$.
Notice that $|B_\tau^*| = |B|$ if $d=2$.

\section{Extensions to complex-valued electric potentials}\label{Sec.AB}
%
After establishing the crucial identity of Lemma~\ref{Lem.crucial},
one of the next steps of~\cite{FKV} to deal with it
was to use a weighted Hardy inequality 
\begin{equation}\label{eq:hardyweight}
  \forall\psi\in C^\infty_0(\mathbb R^d) \,, \qquad
  \int r \, |\nabla \psi|^2 \geq \frac{(d-1)^2}{4}
  \int\frac{|\psi|^2}{r}
\end{equation}
together with the diamagnetic inequality~\eqref{dia}
and to replace the first two terms in the round brackets 
on the left-hand side of~\eqref{crucial} by the lower bound
\begin{equation}\label{problem}
  \int r \, |\nabla_{\!A} u^-|^2
  - \frac{d-1}{2} \int\frac{|u^-|^2}{r}
  \geq 
  \frac{d-3}{d-1} \int r \, |\nabla_{\!A} u^-|^2 
  \,.
\end{equation}
Although~\eqref{eq:hardyweight} is valid also for $d=2$,
the lower bound~\eqref{problem} is \emph{negative}, 
which spoils the subsequent argument
(\cf~the procedure above Theorem~\ref{Thm.multi}).
This is the reason why it is not immediate, in two dimensions, 
to use formula~\eqref{crucial}
for ensuring the absence of eigenvalues 
with \emph{non-zero imaginary part}
(unless~$H_{A,V}$ is self-adjoint).
Notice, however, 
that a condition excluding real eigenvalues even if~$H_{A,V}$
is non-self-adjoint is easy to obtain in the same way as above, 
because then the troublesome term represented by the round brackets 
on the left-hand side of~\eqref{crucial} is not present.

Here we present several alternative ways how to use~\eqref{crucial}
in order to guarantee the total absence of eigenvalues 
even if~$H_{A,V}$ is not self-adjoint. 
Since the alternative approaches are not needed in higher dimensions,
in this section we again restrict to the two-dimensional situation.

First of all, we rewrite~\eqref{crucial} as follows ($d=2$): 
\begin{eqnarray}\label{crucial.bis}
\lefteqn{
  \int |\nabla_{\!A} u^-|^2
  + \frac{|\Im\lambda|}{(\Re\lambda)^{1/2}}
  \left(
  \int r \, |\nabla_{\!A} u^-|^2
  - \frac{1}{2} \int\frac{|u^-|^2}{r}
  + \int r \, \Re V_+ \, |u^-|^2
  \right)
} 
  \nonumber \\
  && = -2 \Im \int r \, B_\tau^* \cdot u^- \overline{\nabla_{\!A}u^-} 
  + \int \partial_r(r \, \Re V^{(1)}) \, |u^-|^2
  - \int \Re V^{(2)} \, |u^-|^2
  \nonumber \\
  && \quad
  - 2 \Re \int r \, \Re V^{(2)} \, u^- \overline{\partial_r^{A}u^-} 
  + 2 \Im \int r \, \Im V \, u^- \overline{\partial_r^{A}u^-} 
  + \frac{|\Im\lambda|}{(\Re\lambda)^{1/2}} \int r \, \Re V_- \, |u^-|^2
  \,,
\end{eqnarray}
\ie~we keep on the left-hand side just the positive part of~$\Re V$.
As above, the terms multiplied by $|\Im\lambda|/(\Re\lambda)^{1/2}$
are not present if $\lambda=0$.
More generally, if $\lambda=0$ or $\Im\lambda=0$,
the identity~\eqref{crucial.bis} coincides with~\eqref{crucial.ss}
and the sufficient condition of Theorem~\ref{Thm} already
implies the absence of such eigenvalues.
Below we derive worth sufficient conditions to cover 
the case of complex eigenvalues as well.

\subsection{Employing the positivity of 
the real part of the electric field}\label{Sec.app1}
An obvious condition to make the round brackets 
on the left-hand side of~\eqref{crucial.bis} non-negative
is to require
\begin{equation}\label{obvious}
  \frac{1}{2} \int\frac{|\psi|^2}{r}
  \leq 
  \int r \, |\nabla_{\!A} \psi|^2
  + \int r \, \Re V_+ \, |\psi|^2 
\end{equation}
for every $\psi \in C_0^\infty(\Real^d)$.
Recalling~\eqref{eq:hardyweight},
it can be satisfied provided that
we require for instance the pointwise bound
\begin{equation}\label{pointwise}
  \Re V_+ \geq \frac{1}{4r^2}
  \,.
\end{equation}

Having ensured the non-negativity of the round brackets 
on the left-hand side of~\eqref{crucial.bis},
the terms on the right-hand side can be handled as above.
It is only important to comment on how to get rid of
the energy dependent fraction at the 
last term on the right-hand side of~\eqref{crucial.bis}.	
We proceed as follows:
\begin{eqnarray}\label{proceed}
  \frac{|\Im\lambda|}{(\Re\lambda)^{1/2}} \int r \, \Re V_- \, |u^-|^2
  &\leq& 
  \frac{|\Im\lambda|}{(\Re\lambda)^{1/2}}
  \sqrt{\int r^2 \, |\Re V_-|^2 \, |u^-|^2}
  \, \sqrt{\int |u^-|^2}
  \nonumber \\
  &\leq& 
  \frac{|\Im\lambda|^{1/2}}{(\Re\lambda)^{1/2}}
  \sqrt{\int r^2 \, |\Re V_-|^2 \, |u^-|^2}
  \, \sqrt{\int |\Im V| |u^-|^2}
  \nonumber \\
  &\leq& 
  \sqrt{\int r^2 \, |\Re V_-|^2 \, |u^-|^2}
  \, \sqrt{\int |\Im V| \, |u^-|^2}
  \,,
\end{eqnarray}
where the first estimate is due the Schwarz inequality,
the second estimate follows from~\eqref{G2} 
with a constant choice for the multiplier~$G_2$
and the last estimate is implied by the restriction
to the sector~\eqref{sector}.	
Consequently, from~\eqref{crucial.bis}
we deduce the bound 
\begin{eqnarray*}
  \int |\nabla_{\!A} u^-|^2
  &\leq& 2 \, \sqrt{\int r^2 \, |B|^2 \, |u^-|^2 } \, 
  \sqrt{\int |\nabla_{\!A}u^-|^2} 
  + \int [\partial_r(r \, \Re V^{(1)})]_+ \, |u^-|^2
  \nonumber \\
  && 
  + \int |\Re V^{(2)}| \, |u^-|^2
  + 2 \, \sqrt{\int r^2 \, |\Re V^{(2)}|^2 \, |u^-|^2}
  \, \sqrt{\int |\nabla_{\!A}u^-|^2}
  \\
  && 
  + 2 \, \sqrt{\int r^2 \, |\Im V|^2 \, |u^-|^2}
  \, \sqrt{\int |\nabla_{\!A}u^-|^2}
  + \sqrt{\int r^2 \, |\Re V_-|^2 \, |u^-|^2}
  \, \sqrt{\int |\Im V| \, |u^-|^2}
  \,.
  \\
  &\leq&
  (b_1 +b_2^2 + b_3^2 + b_4 + b_5 + b_6 a_2) 
  \int |\nabla_{\!A}u^-|^2
  \,, 
\end{eqnarray*}
where the last inequality follows by~\eqref{Ass1.bis}  
and conditions
\begin{equation}\label{Ass2.nsa}
\begin{aligned}
  \int 4 \, r^2 \, |B|^2 \, |\psi|^2 
  &\leq b_1^2 \int |\nabla_{\!A} \psi|^2 \,,
  &&&
  \int [\partial_r(r \, \Re V^{(1)})]_+ \, |\psi|^2
  &\leq b_2^2 \int |\nabla_{\!A} \psi|^2 \,,
  \\
  \int |\Re V^{(2)}| \, |\psi|^2
  &\leq b_3^2 \int |\nabla_{\!A} \psi|^2 \,,
  &&&
  \int 4 \, r^2 \, |\Re V^{(2)}|^2 \, |\psi|^2
  &\leq b_4^2 \int |\nabla_{\!A} \psi|^2 \,,
  \\
  \int 4 \, r^2 \, |\Im V| \, |\psi|^2
  &\leq b_5^2 \int |\nabla_{\!A} \psi|^2 \,,
  &&&
  \int r^2 \, |\Re V_-|^2 \, |\psi|^2
  &\leq b_6^2 \int |\nabla_{\!A} \psi|^2 \,,
\end{aligned}
\end{equation}
assumed to be valid for every $\psi \in C_0^\infty(\Real^d)$.
Making the hypothesis
\begin{equation}\label{Ass.b.nsa}
  b_1 +b_2^2 + b_3^2 + b_4 + b_5 + b_6 a_2 < 1
  \,,
\end{equation}
we conclude with $u=0$ as above.

We summarise the result of this subsection in the following theorem.
\begin{Theorem}\label{Thm.nsa1}
Let $A \in L_\mathrm{loc}^2(\Real^2;\Real^2)$ 
be such that $B \in L_\mathrm{loc}^2(\Real^2;\Real)$.
Suppose that $V \in L^1(\Real^2;\Com)$
admits the decomposition $\Re V = \Re V^{(1)} + \Re V^{(2)}$ 
where $\Re V^{(1)} \in W^{1,1}_\mathrm{loc}(\Real^d)$ 
and $\Re V^{(2)} \in L_\mathrm{loc}^2(\Real^d)$.
Assume that there exist numbers $a_1,a_2,b_1,b_2,b_3,b_4,b_5,b_6 \in [0,1)$ 
satisfying~\eqref{Ass.b.nsa}
such that~\eqref{Ass1.bis} and~\eqref{Ass2.nsa} hold. 
Moreover, assume~\eqref{obvious}.
Then $H_{A,V}$ has no eigenvalues, \ie\
$\sigma_\mathrm{p}(H_{A,V}) = \varnothing$.
\end{Theorem}
\begin{Remark}
Alternatively, keeping the last term 
from the right-hand side of~\eqref{crucial.bis}
in the round brackets on left-hand side,
one can require the condition 
\begin{equation*}
  \frac{1}{2} \int\frac{|\psi|^2}{r}
  + \int r \, \Re V_- \, |\psi|^2 
  \leq 
  \int r \, |\nabla_{\!A} \psi|^2
  + \int r \, \Re V_+ \, |\psi|^2 
\end{equation*}
instead of~\eqref{obvious}.
Recalling~\eqref{eq:hardyweight},
it can be satisfied provided that
we require for instance the pointwise bound
\begin{equation*}
  \Re V \geq \frac{1}{4r^2} 
  \,.
\end{equation*}
Then the conditions on the last line of~\eqref{Ass2.nsa}
can be ignored and one can take $b_5,b_6=0$ in~\eqref{Ass.b.nsa}.
\end{Remark}

\subsection{A more robust approach}\label{Sec.app2}
Denote by
$D_R:=\{x\in\Real^2: |x| < R\}$ the open disk of radius $R>0$.
The main ingredient of this subsection is the following
Hardy-Poincar\'e-type inequality.
\begin{Lemma}\label{Lem.HP}
One has
\begin{equation}\label{HP}
  \forall \psi \in W_0^{1,2}(D_R)
  \,, \qquad
  \int_{D_R} |\nabla\psi(x)|^2 \, \der x
  \geq \frac{1}{4 R}
  \int_{D_R} \frac{|\psi(x)|^2}{|x|} \, \der x
  \,.
\end{equation}
\end{Lemma}
\begin{proof}
For every $f \in C^1([0,R])$ such that $f(R)=0$, 
we have
\begin{align*}
  \int_0^R \frac{|f(r)|^2}{r} \, r \, \der r
  &= \int_0^R |f(r)|^2 \, r' \, \der r
  \\
  &= - 2 \int_0^R \Re\big[\overline{f(r)} f'(r)\big] 
  \, r \, \der r
  \\
  &\leq 2 \,
  \sqrt{ \int_0^R |f'(r)|^2 \, r  \, \der r } \,
  \sqrt{ \int_0^R |f(r)|^2 \, r  \, \der r } 
  \\
  &\leq  2 R  \,
  \sqrt{ \int_0^R |f'(r)|^2 \, r  \, \der r } \,
  \sqrt{ \int_0^R \frac{|f(r)|^2}{r} \, r  \, \der r } 
  \,,
\end{align*}
and therefore
$$
  \int_0^R |f'(r)|^2 \, r  \, \der r
  \geq \frac{1}{4 R}
  \int_0^R \frac{|f(r)|^2}{r} \, r  \, \der r
  \,.	
$$
This one-dimensional inequality implies~\eqref{HP}
after expressing the gradient in spherical coordinates
and by neglecting the angular component. 
\end{proof}

Given two positive numbers $R_1 < R_2$,
let $\eta:[0,\infty) \to [0,1]$ be such that $\eta=1$ on $[0,R_1]$,
$\eta=0$ on $[R_2,\infty)$ and $\eta(r)=(R_2-r)/(R_2-R_1)$ 
for $r \in (R_1,R_2)$.
We denote by the same symbol~$\eta$ 
the radial function $\eta \circ r : \Real^2 \to [0,1]$.
Now, writing $u^- = \eta u^- + (1-\eta) u^-$
and using Lemma~\ref{Lem.HP}, 
we estimate the troublesome term of~\eqref{crucial.bis} as follows:
\begin{align*}
  \int \frac{|u^-|^2}{r} 
  &\leq 
  2 \int \frac{|\nabla  u^-|^2}{r}  + 2 \int \frac{|(1-\eta)u^-|^2}{r} 
  \\
  &\leq 
  8 R_2 \int |\nabla(\eta u^-)|^2
  + \frac{2}{R_1} \int |u^-|^2
  \\
  &\leq
  16 R_2 \int |\nabla u^-|^2
  + \frac{16 R_2}{(R_2-R_1)^2} \int |u^-|^2
  + \frac{2}{R_1} \int |u^-|^2 
  \,.
\end{align*}
Choosing $R_2 := \epsilon \, (\Re\lambda)^{1/2}/|\Im\lambda|$
and $R_1 := R_2/2$ with any positive number~$\epsilon$,
we get
\begin{align*}
  \frac{|\Im\lambda|}{(\Re\lambda)^{1/2}}
  \int \frac{|u^-|^2}{r} 
  &\leq 16 \, \epsilon \int |\nabla u^-|^2
  + \frac{68}{\epsilon} \frac{|\Im\lambda|^2}{\Re\lambda}
  \int |u^-|^2 
  \\
  &\leq 16 \, \epsilon \int |\nabla u^-|^2
  + \frac{68}{\epsilon} \frac{|\Im\lambda|}{\Re\lambda}
  \int |\Im V| \, |u^-|^2 
  \\
  &\leq 16 \, \epsilon \int |\nabla u^-|^2
  + \frac{68}{\epsilon}  
  \int |\Im V| \, |u^-|^2 
  \,,
\end{align*}
where the second estimate follows from~\eqref{G2} 
with a constant choice for the multiplier~$G_2$
and the last estimate is implied by the restriction
to the sector~\eqref{sector}.	
Using additionally~\eqref{eq:hardyweight}, 
we thus deduce from~\eqref{crucial.bis} the crucial inequality
\begin{eqnarray}\label{crucial.robust}
\lefteqn{
  (1-4\epsilon) \int |\nabla_{\!A} u^-|^2
  - \frac{17}{\epsilon} \int |\Im V| \, |u^-|^2 
  + \frac{|\Im\lambda|}{(\Re\lambda)^{1/2}}
  \int r \, \Re V_+ \, |u^-|^2  
} 
  \nonumber \\
  && \leq -2 \Im \int r \, B_\tau^* \cdot u^- \overline{\nabla_{\!A}u^-} 
  + \int \partial_r(r \, \Re V^{(1)}) \, |u^-|^2
  - \int \Re V^{(2)} \, |u^-|^2
  \nonumber \\
  && \quad
  - 2 \Re \int r \, \Re V^{(2)} \, u^- \overline{\partial_r^{A}u^-} 
  + 2 \Im \int r \, \Im V \, u^- \overline{\partial_r^{A}u^-} 
  + \frac{|\Im\lambda|}{(\Re\lambda)^{1/2}} \int r \, \Re V_- \, |u^-|^2
  \,.
\end{eqnarray}

Now, putting the second term from the left-hand side 
of~\eqref{crucial.robust} to the right-hand side,
neglecting the last term on the left-hand side  
and treating the terms on the right-hand side as 
in Section~\ref{Sec.app1} (see particularly~\eqref{proceed}),
we get
\begin{eqnarray}\label{robust.estimate}
  (1-4\epsilon) \int |\nabla_{\!A} u^-|^2
  &\leq& 2 \, \sqrt{\int r^2 \, |B|^2 \, |u^-|^2 } \, 
  \sqrt{\int |\nabla_{\!A}u^-|^2} 
  + \int [\partial_r(r \, \Re V^{(1)})]_+ \, |u^-|^2
  \nonumber \\
  && 
  + \int |\Re V^{(2)}| \, |u^-|^2
  + 2 \, \sqrt{\int r^2 \, |\Re V^{(2)}|^2 \, |u^-|^2}
  \, \sqrt{\int |\nabla_{\!A}u^-|^2}
  \nonumber \\
  && 
  + 2 \, \sqrt{\int r^2 \, |\Im V|^2 \, |u^-|^2}
  \, \sqrt{\int |\nabla_{\!A}u^-|^2}
  + \sqrt{\int r^2 \, |\Re V_-|^2 \, |u^-|^2}
  \, \sqrt{\int |\Im V| \, |u^-|^2}
  \nonumber \\
  &&
  + \frac{17}{\epsilon} \int |\Im V| \, |u^-|^2 
  \nonumber \\
  &\leq&
  \left(
  b_1 +b_2^2 + b_3^2 + b_4 + b_5 + b_6 a_2 + \frac{17}{\epsilon} a_2^2
  \right) 
  \int |\nabla_{\!A}u^-|^2
  \,, 
\end{eqnarray}
where the last inequality follows by~\eqref{Ass1.bis}  
and conditions~\eqref{Ass2.nsa}.
Making the hypothesis 
\begin{equation}\label{Ass.robust}
  b_1 +b_2^2 + b_3^2 + b_4 + b_5 + b_6 a_2 
  + \frac{17}{\epsilon} a_2^2 + 4 \epsilon < 1
  \,,
\end{equation}
we therefore conclude with $u = 0$ as above.

We summarise the result of this subsection in the following theorem.
\begin{Theorem}\label{Thm.robust}
Let $A \in L_\mathrm{loc}^2(\Real^2;\Real^2)$ 
be such that $B \in L_\mathrm{loc}^2(\Real^2;\Real)$.
Suppose that $V \in L^1(\Real^2;\Com)$
admits the decomposition $\Re V = \Re V^{(1)} + \Re V^{(2)}$ 
where $\Re V^{(1)} \in W^{1,1}_\mathrm{loc}(\Real^d)$ 
and $\Re V^{(2)} \in L_\mathrm{loc}^2(\Real^d)$.
Assume that there exist numbers 
$\epsilon,a_1,a_2,b_1,b_2,b_3,b_4,b_5,b_6 \in [0,1)$ 
satisfying~\eqref{Ass.b.nsa}
such that~\eqref{Ass1.bis} and~\eqref{Ass2.nsa} hold. 
Then $H_{A,V}$ has no eigenvalues, \ie\
$\sigma_\mathrm{p}(H_{A,V}) = \varnothing$.
\end{Theorem}

\subsection{The Aharonov-Bohm potential}\label{Sec.app3}
Finally, we consider the special case of 
the singular \emph{Aharonov-Bohm potential}
\begin{equation}\label{AB}
  A(x) := (-\sin\theta,\cos\theta) \, \frac{\alpha(\theta)}{r}
  \,,
\end{equation}
where $(x_1,x_2)=(r\cos\theta,r\sin\theta)$
is the parameterisation via polar coordinates,
$r \in (0,\infty)$, $\theta \in [0,2\pi)$,
and $\alpha : [0,2\pi) \to \Real$ is 
an arbitrary bounded function.
In this case, the magnetic field~$B$ equals zero 
everywhere except for $x=0$. In fact 
\begin{equation}\label{AB.field}
   B = 2\pi \, \overline{\alpha} \, \delta
\end{equation}
in the sense of distributions, 
where~$\delta$ is the Dirac delta function and 
$$
  \overline{\alpha} 
  := \frac{1}{2\pi} \int_0^{2\pi} \alpha(\theta) \, \der\theta
$$
has the physical meaning of the total magnetic flux. 
We notice that~$A$ can be gauged out
whenever~$\overline{\alpha}$ is an integer.
To measure the strength of the Aharonov-Bohm field,
we introduce the distance of the total magnetic flux~$\overline{\alpha}$
to the set of integers
\begin{equation}\label{beta}
  \beta := \dist(\overline{\alpha},\Int)
\end{equation}
and assume $\overline{\alpha} \not\in \Int$,
so that $\beta \in (0,1/2]$.

Since $A \not\in \sii_\mathrm{loc}(\Real^2;\Real^2)$,
the Aharonov-Bohm potential does not satisfy 
our standing regularity assumption.
In the case~\eqref{AB}, we still understand~$H_{A,V}$ 
as the operator associated with the form~\eqref{form},
but now with the form core $C_0^\infty(\Real^2)$ 
being replaced by $C_0^\infty(\Real^2\setminus \{0\})$. 
To get an m-sectorial operator, 
we assume the validity of~\eqref{Ass1.bis} 
for all $\psi \in C_0^\infty(\Real^2\setminus \{0\})$ now.
If $\beta=0$, then~$H_{A,V}$ is unitarily equivalent 
to the magnetic-free operator~$H_{0,V}$. 
	
Following the proof of Lemma~\ref{Lem.crucial},
it can be shown that the identity	
\begin{eqnarray}\label{crucial.AB}
\lefteqn{
  \int |\nabla_{\!A} u^-|^2
  + \frac{|\Im\lambda|}{(\Re\lambda)^{1/2}}
  \left(
  \int r \, |\nabla_{\!A} u^-|^2
  - \frac{1}{2} \int\frac{|u^-|^2}{r}
  + \int r \, \Re V_+ \, |u^-|^2
  \right)
} 
  \nonumber \\
  && = \int \partial_r(r \, \Re V^{(1)}) \, |u^-|^2
  - \int \Re V^{(2)} \, |u^-|^2
  \nonumber \\
  && \quad
  - 2 \Re \int r \, \Re V^{(2)} \, u^- \overline{\partial_r^{A}u^-} 
  + 2 \Im \int r \, \Im V \, u^- \overline{\partial_r^{A}u^-} 
  + \frac{|\Im\lambda|}{(\Re\lambda)^{1/2}}
  \int r \, \Re V_- \, |u^-|^2
\end{eqnarray}
holds, provided that the solution of $H_{A,V}u=\lambda u$
satisfies conditions~\eqref{crucial.cond} without the first term
(containing~$B_\tau^*$).
Formally, \eqref{crucial.AB}~follows after plugging~\eqref{AB.field}
into~\eqref{crucial}.

The principal idea of this subsection is that 
the singular magnetic Hardy-type inequality 
due to Laptev and Weidl (see~\cite[Thm.~3]{Laptev-Weidl_1999})
\begin{equation}\label{Hardy.AB}
  \forall \psi\in C_0^\infty(\Real^2\setminus \{0\})
  \,, \qquad
  \int_{\Real^2} |\nabla_{\!A}\psi(x)|^2 \, \der x
  \geq \beta^2
  \int_{\Real^2} \frac{|\psi(x)|^2}{|x|^2} \, \der x
\end{equation}
holds true,
where~$\beta$ is introduced in~\eqref{beta}.
We also use its weighted variant included in the following lemma,
which is a magnetic improvement upon~\eqref{eq:hardyweight}.
\begin{Lemma}\label{Lem.Hardy.AB.weighted}
Let $A$ be given by~\eqref{AB}. Then one has
\begin{equation}\label{Hardy.AB.weighted}
  \forall \psi\in C_0^\infty(\Real^2\setminus \{0\})
  \,, \qquad
  \int_{\Real^2} |x| \, |\nabla_{\!A}\psi(x)|^2 \, \der x
  \geq \left(\frac{1}{4}+\beta^2\right)
  \int_{\Real^2} \frac{|\psi(x)|^2}{|x|} \, \der x
  \,.
\end{equation}
\end{Lemma}
\begin{proof}
Passing to polar coordinates, we have
$$
\begin{aligned}
  \int_{\Real^2} |x| \, |\nabla_{\!A}\psi(x)|^2 \, \der x
  &= \int_0^{2\pi} \int_0^\infty |\partial_r \phi(r,\theta)|^2 
  \, r^2 \, \der r \, \der\theta
  + \int_0^{2\pi} \int_0^\infty 
  |(\partial_\theta + i\alpha(\theta))\phi(r,\theta)|^2 
  \, r^2 \, \der r \, \der\theta
  \\
  &\geq \frac{1}{4} 
  \int_0^{2\pi} \int_0^\infty |\phi(r,\theta)|^2 
  \, r^2 \, \der r \, \der\theta
  + \beta^2 \int_0^{2\pi} \int_0^\infty |\phi(r,\theta)|^2 
  \, r^2 \, \der r \, \der\theta
  \,,
\end{aligned}
$$
where~$\phi(r,\theta) := \psi(r\cos\theta,r\sin\theta)$.
Here the first integral on the right-hand side is 
estimated by a classical weighted one-dimensional Hardy inequality
(which is actually behind the proof of~\eqref{eq:hardyweight}),
while the bound on the second integral employs that the first eigenvalue
of the operator $[-i\partial_\theta + \alpha(\theta)]^2$ in $\sii((0,2\pi))$,
subject to periodic boundary conditions, equals~$\beta^2$.
\end{proof}

Lemma~\ref{Lem.Hardy.AB.weighted} enables us 
to handle the troublesome term 
on the left-hand side of~\eqref{crucial.AB} as follows.
Given any positive number~$\delta$, 
we write
\begin{align*} 
  - \int r \, |\nabla_{\!A} u^-|^2
  + \frac{1}{2} \int\frac{|u^-|^2}{r} 
  &\leq \left(\frac{1}{4} - \beta^2\right) 
  \int\frac{|u^-|^2}{r} 
  \\
  &= \left(\frac{1}{4} - \beta^2\right)
  \int_{D_\delta}\frac{|u^-|^2}{r} 
  + \left(\frac{1}{4} - \beta^2\right)
  \int_{\Real^2 \setminus \overline{D_\delta}}\frac{|u^-|^2}{r} 
  \\
  &\leq \left(\frac{1}{4} - \beta^2\right) \delta
  \int \frac{|u^-|^2}{r^2}
  + \left(\frac{1}{4} - \beta^2\right) \frac{1}{\delta}
  \int |u^-|^2 
  \\
  &\leq \left(\frac{1}{4} - \beta^2\right) \delta
  \int \frac{|u^-|^2}{r^2}
  + \left(\frac{1}{4} - \beta^2\right) \frac{1}{\delta\, |\Im\lambda|}
  \int |\Im V| \, |u^-|^2 
  \\
  &\leq \left(\frac{1}{4} - \beta^2\right) \frac{\delta}{\beta^2} 
  \int |\nabla_{\!A} u^-|^2 
  +  \left(\frac{1}{4} - \beta^2\right) \frac{1}{\delta\, |\Im\lambda|}
  \int |\Im V| \, |u^-|^2 
  \,,
\end{align*}
where the first inequality is due to~\eqref{Hardy.AB.weighted},
estimates in the second inequality are elementary,
the third estimate follows from~\eqref{G2} 
with a constant choice for the multiplier~$G_2$
and the last inequality employs~\eqref{Hardy.AB}.
Choosing 
$\delta := \beta^2 \, \epsilon \, (\Re\lambda)^{1/2} / |\Im\lambda|$
with any positive~$\epsilon$ and using~\eqref{sector},
we therefore get
\begin{equation*} 
  \frac{|\Im\lambda|}{(\Re\lambda)^{1/2}}
  \left(
  - \int r \, |\nabla_{\!A} u^-|^2
  + \frac{1}{2} \int\frac{|u^-|^2}{r} 
  \right)
  \leq  
  \left(\frac{1}{4} - \beta^2\right) \epsilon \int |\nabla_{\!A} u^-|^2 
  + \left(\frac{1}{4} - \beta^2\right)
  \frac{1}{\beta^2\,\epsilon} \int |\Im V| \, |u^-|^2 
  \,.
\end{equation*}
Neglecting the last term on the left-hand side of~\eqref{crucial.AB}
and treating the other terms as 
in the first inequality of~\eqref{robust.estimate},
we arrive at
\begin{eqnarray*} 
  \int |\nabla_{\!A} u^-|^2
  &\leq&  
  \int [\partial_r(r \, \Re V^{(1)})]_+ \, |u^-|^2
  \nonumber \\
  && 
  + \int |\Re V^{(2)}| \, |u^-|^2
  + 2 \, \sqrt{\int r^2 \, |\Re V^{(2)}|^2 \, |u^-|^2}
  \, \sqrt{\int |\nabla_{\!A}u^-|^2}
  \nonumber \\
  && 
  + 2 \, \sqrt{\int r^2 \, |\Im V|^2 \, |u^-|^2}
  \, \sqrt{\int |\nabla_{\!A}u^-|^2}
  + \sqrt{\int r^2 \, |\Re V_-|^2 \, |u^-|^2}
  \, \sqrt{\int |\Im V| \, |u^-|^2}
  \nonumber \\
  &&
  + \left(\frac{1}{4} - \beta^2\right) \epsilon \int |\nabla_{\!A} u^-|^2 
  + \left(\frac{1}{4} - \beta^2\right)
  \frac{1}{\beta^2\,\epsilon} \int |\Im V| \, |u^-|^2
  \nonumber \\
  &\leq&
  \left[
  b_2^2 + b_3^2 + b_4 + b_5 + b_6 a_2 
  + \left(\frac{1}{4} - \beta^2\right) \epsilon
  + \left(\frac{1}{4} - \beta^2\right)
  \frac{1}{\beta^2 \, \epsilon} a_2^2
  \right] 
  \int |\nabla_{\!A}u^-|^2
  \,. 
\end{eqnarray*}
Here the last inequality follows by~\eqref{Ass1.bis}  
and conditions~\eqref{Ass2.nsa} that are assumed to be valid
for all $\psi \in C_0^\infty(\Real^2\setminus\{0\})$.
The first integral in~\eqref{Ass2.nsa} is interpreted as zero,
so that this condition is always satisfied 
for the Aharonov-Bohm field~\eqref{AB}.
Making the hypothesis 
\begin{equation}\label{Ass.AB}
  b_2^2 + b_3^2 + b_4 + b_5 + b_6 a_2 
  + \left(\frac{1}{4} - \beta^2\right) \epsilon
  + \left(\frac{1}{4} - \beta^2\right)
  \frac{1}{\beta^2 \, \epsilon} a_2^2
  < 1
  \,,
\end{equation}
we therefore conclude with $u = 0$ as above.

We summarise the result of this subsection in the following theorem.
\begin{Theorem}\label{Thm.AB}
Let the vector potential~$A$ be given by~\eqref{AB}
with $\overline{\alpha} \not\in \Int$
and suppose that the scalar potential $V \in L^1(\Real^2;\Com)$
admits the decomposition $\Re V = \Re V^{(1)} + \Re V^{(2)}$ 
where $\Re V^{(1)} \in W^{1,1}_\mathrm{loc}(\Real^d)$ 
and $\Re V^{(2)} \in L_\mathrm{loc}^2(\Real^d)$.
Assume that there exist numbers 
$\epsilon,a_1,a_2,b_2,b_3,b_4,b_5,b_6 \in [0,1)$ 
satisfying~\eqref{Ass.AB}
such that~\eqref{Ass1.bis} and~\eqref{Ass2.nsa} hold
for all $\psi \in C_0^\infty(\Real^2\setminus\{0\})$.
Then $H_{A,V}$ has no eigenvalues, \ie\
$\sigma_\mathrm{p}(H_{A,V}) = \varnothing$.
\end{Theorem}

In particular, in the electric-free case $V=0$,
we recover the well known result that the Aharonov-Bohm
Laplacian~$H_{A,0}$ possesses no eigenvalues.
Notice also that the sufficient conditions of Theorem~\ref{Thm.AB}
substantially simplify in the regime when the Aharonov-Bohm
field is the strongest, \ie~$\beta=1/2$.


\subsection*{Acknowledgments}
We are grateful to Timo Weidl for valuable suggestions.
The research of D.K.\ was partially supported by FCT (Portugal)
through project PTDC/MAT-CAL/4334/2014.
The research of L.V.\ was partially supported by
ERCEA Advanced Grant 669689-HADE, MTM2014-53145-P, and IT641-13. 

\vfill
%
\renewcommand\refname{\normalsize References}
{\small
\bibliography{bib}
\bibliographystyle{amsplain}
}

\end{document}